\documentclass[12pt, a4paper]{amsart}
\usepackage[english]{babel}
\usepackage{amssymb,amsmath,amsthm,mathrsfs,bm,nicefrac}
\usepackage{color}
\usepackage[margin=1in]{geometry}
\usepackage{stmaryrd}

\newtheorem{theorem}[subsection]{Theorem}

\newtheorem{lemma}[subsection]{Lemma}

\newtheorem{proposition}[subsection]{Proposition}
\newtheorem{corollary}[subsection]{Corollary}

\newtheorem*{claim*}{Claim}
\newtheorem*{theorem*}{Theorem}

\def\bal{\begin{aligned}}
\def\eal{\end{aligned}}
\def\be{\begin{equation}\label}
\def\ee{\end{equation}}
\def\bcs{\begin{cases}}
\def\ecs{\end{cases}}
\def\={\;=\;}
\def\+{\,+\,}
\def\-{\,-\,}

\def\Z{{\mathbb Z}}

\def\R{{\mathbb R}}

\def\ord{\mathrm{ord}}

\def\scrO{\mathscr{O}}
\def\scrY{\mathscr{Y}}

\def\cartier{\mathscr{C}_p}

\def\v#1{{\bf #1}}

\def\is{\equiv}
\def\mod#1{({\rm mod}\ #1)}
\def\hat{\widehat}

\title{A matrix version of Dwork's congruences}
\author{Frits Beukers}

\address{Utrecht University }
\email{f.beukers@uu.nl}

\thanks{
Work supported by the Netherlands Organisation for Scientific Research
(NWO), grant TOP1EW.15.313.}

\begin{document}
\maketitle

\begin{abstract}
In this article we give an example of a matrix version of the
famous congruence for hypergeometric functions found by Dwork in
'p-adic cycles'. 
\end{abstract}

\section{Introduction}
In this paper we shall deal with results of the following type.
Let $F(t)$ be an infinite power series with constant term 1
and coefficients in $\Z_p$, the $p$-adic numbers. Denote by $F(t)_m$ its $m$-th
truncation, i.e all terms of degree $\ge m$ are deleted. We shall be interested
whether there are hypergeometric series $F(t)$ for which 

\begin{equation}\label{rank1}
\frac{F(t)}{F(t^p)}\is \frac{F(t)_{p^s}}{F(t^p)_{p^{s-1}}}\mod{p^s}
\end{equation}
for all $s\ge1$. 
The first such result was given by Dwork in 'p-adic cycles', \cite{dwork69}, for
the case of $F(t)=F(\nicefrac12,\nicefrac12,1|t)$. We shall be interested in 
congruences of type (\ref{rank1}) for other hypergeometric series. Contrary to
my initial expectations the number of such results in the literature seems to
be very limited. The goal of the present paper is to provide some more hypergeometric
examples of (\ref{rank1}) and present an example of a matrix version. The underlying
machinery to produce such results are the papers \cite{BV19I}
and \cite{BV19II} (Dwork crystals I and II) in which we give an elementary approach
to the construction of the so-called unit root crystal in Dwork's $p$-adic theory
of zeta-functions of algebraic varieties. The present paper can be seen as an illustration
of the results in \cite{BV19I} and \cite{BV19II}. In Section \ref{MUM} we give some
straightforward rank one applications. In Section \ref{matrixexample} we present
our main result, Theorem \ref{main}, containing an example of a matrix version
of Dwork's congruence. Its proof
requires some ideas in addition to \cite{BV19I} and \cite{BV19II}.

{\bf Acknowledgement} I would like to thank Ling Long for our discussions
which gave rise to this paper. 

\section{Summary of \cite{BV19I} and \cite{BV19II}}\label{theory}
Let $R$ be a characteristic zero domain and $p$ an odd prime such that
$\cap_{s\ge1}p^sR=\{0\}$. Suppose that $R$ is $p$-adically complete.
Let $f \in R[x_1^{\pm1},\ldots,x_n^{\pm 1}]$ be a Laurent polynomial
and $\Delta \subset \R^n$ its Newton polytope. Let $\Delta^\circ$ be its interior. 
Consider the $R$-module $\Omega_f^\circ$ of differential forms generated over $R$ by
\[
\omega_{\v u}:= (k-1)!\frac{\v x^{\v u}}{f(\v x)^{k}}\frac{dx_1}{x_1}\wedge
\cdots\wedge\frac{dx_n}{x_n},\quad \v u\in k\Delta^\circ
\]
for all $k\ge1$. Contrary to \cite{BV19I} and \cite{BV19II} we have now written
the elements of $\Omega_f$ as differential forms. Let us abbreviate
$\frac{dx_1}{x_1}\wedge\cdots\wedge\frac{dx_n}{x_n}$ by $\frac{d\v x}{\v x}$.

Differential forms in $\Omega_f^\circ$ can be expanded as formal Laurent series. To that
end we choose a vertex $\v b$ of $\Delta$ and obtain a Laurent expansion with support
in $C(\Delta-\v b)$, the positive cone generated by the vectors in $\Delta-\v b$,
and coefficients in $R$ of the form
\[
\sum_{\v k\in C(\Delta-\v b)}a_{\v k}\v x^{\v k}\frac{d\v x}{\v x},\quad a_{\v k} \in R.
\]
We denote such forms by $\Omega_{\rm formal}$. The exact forms are denoted by
$d\Omega_{\rm formal}$. We call them formally exact forms and they are characterized
by the following lemma of Katz.

\begin{lemma}\label{exactcondition}
A series $\sum_{\v k\in C(\Delta-\v b)}a_{\v k}\v x^{\v k}\frac{d\v x}{\v x}$
is a formal derivative if and only if 
\[
a_{\v k}\is0\mod{p^{\ord_p(\v k)}}\quad\mbox{for all $\v k$}.
\] 
Here $\ord_p(k)$ denotes the $p$-adic valuation of $k$ and 
$\ord_p(\v k)=\min(\ord_p(k_1),\ldots,\ord_p(k_n))$. 
\end{lemma}

We define the Cartier operator $\cartier$ on $\Omega_{\rm formal}$ by
\begin{equation}\label{cartier-on-series-def}
\cartier\left(\sum_{\v k}a_{\v k}\v x^{\v k}\frac{d\v x}{\v x}\right):
=\sum_{\v k}a_{p\v k}\v x^{\v k}\frac{d\v x}{\v x}.
\end{equation}
Using $\cartier$ we have an alternative characterization of formally exact forms
which is a direct consequence of Lemma \ref{exactcondition}.

\begin{lemma}\label{exactcondition2}
A series $h\in\Omega_{\rm formal}$ is a formal derivative if and only if
$\cartier^s(h)\is0\mod{p^s}$ for all integers $s\ge1$.
\end{lemma}

When applied to a rational differential form $\cartier$ acts as
$$\cartier\left(S(\v x)\frac{d\v x}{\v x}\right)=
\sum_{\v y:\v y^p=\v x}S(\v y)\frac{d\v y}{\v y}.$$
The summation extends over all $y_i=\zeta_i x_i^{1/p}, i=1,\ldots,n$, where each $\zeta_i$
runs over all $p$-th roots of unity. So we see that $\cartier$ sends rational
differential forms to rational differential forms. Unfortunately, $\Omega_f^\circ$
is not sent to itself. But we have something that comes close. Define
the $p$-adic completion
$$\hat{\Omega}_f^\circ:=\underset{\leftarrow}\lim  \Omega_f^\circ/ p^s \Omega_f^\circ.$$
Fix a Frobenius lift $\sigma$ on $R$: this is a ring endomorphism $\sigma: R \to R$
such that $\sigma(r) \is r^p \mod {p}$ for every $r \in R$. We have

\begin{proposition}\label{cartier-def} 
If $p>2$ then $\cartier(\Omega_f^\circ) \subset \hat \Omega_{f^{\sigma}}^\circ$.
\end{proposition}

The proof is given in \cite[Prop 5]{BV19I} and consists of a straightforward computation
ending with a $p$-adic expansion in $\hat\Omega_{f^\sigma}^\circ$. 

We shall be interested in $U_f^\circ:=\hat\Omega_f^\circ\cap d\Omega_{\rm formal}$. 
These are differential forms that are not necessarily exact but become
exact when embedded in the formal expansions. Katz refers to them as 'forms that
die on formal expansion'. In \cite[Prop 10]{BV19I} we find a characterization 
of the elements of $U_f^\circ$ without any reference to formal expansion.

\begin{proposition}\label{Uf-formal-exact}
With the notations as above we have 
$$
U_f^\circ=\{\omega\in\hat\Omega_f^\circ \;|\; \cartier^s(\omega)\is0
\mod{p^s \hat \Omega_{f^{\sigma^s}}^\circ} \ \mbox{for all $s\ge1$}\}.
$$
\end{proposition}

We now come to one of the main results in \cite[Thm 11]{BV19I}.
Let $h=|\Delta^\circ\cap\Z^n|$.
Define the Hasse-Witt matrix $\beta_p$ as the $h\times h$-matrix  given by
$$(\beta_p)_{\v u,\v v}=\mbox{coefficient of $\v x^{p\v u-\v v}$ of }f(\v x)^{p-1},
\quad \v u,\v v\in\Delta^\circ\cap\Z^n$$

\begin{theorem}
Suppose $\det(\beta_p)$ is invertible in $R$. Then $\Omega_f^\circ/U_f^\circ$ is a
free $R$-module of rank $h$ with basis $\frac{\v x^{\v u}}{f}\frac{d\v x}{\v x}$,
$\v u\in\Delta^\circ\cap\Z^n$. 
\end{theorem}

The remainder of \cite{BV19I} and \cite{BV19II} is then devoted to the construction
of $p$-adic approximations to the $h\times h$-matrix of the Cartier operator.
In \cite{BV19II} we give special attention to those approximations that give rise
to congruences of the form (\ref{rank1}) (in case $h=1$) and higher. 

\section{First examples}\label{MUM}
In \cite{MV16} we find a very general theorem providing congruences of the
form (\ref{rank1}).

\begin{theorem}[Mellit-Vlasenko]\label{MVthm}
Let $g(\v x)\in\Z_p[x_1^{\pm1},\ldots,x_n^{\pm1}]$ be a Laurent polynomial in the
variables $x_1,\ldots,x_n$. Suppose that the Newton polytope $\Delta$ of $g$
has the origin as unique interior lattice point. For every integer $r\ge0$
denote by $f_r$ the constant term of $g(\v x)^r$ and define $F(t)=\sum_{r\ge0}f_rt^r$.
Then the congruences (\ref{rank1}) hold for all $s\ge1$. 
\end{theorem}

In \cite[(7)]{BV19II} there is a stronger result with an entirely different proof.
\begin{theorem}[Beukers-Vlasenko]\label{BVthm}
With the same notations as in Theorem \ref{MVthm} we have
\begin{equation}\label{rank1bis}
\frac{F(t)}{F(t^p)}\is \frac{F(t)_{mp^s}}{F(t^p)_{mp^{s-1}}}\mod{p^s}
\end{equation}
for all $m,s\ge1$.
\end{theorem}

Here is an application. 
\begin{corollary}
Let $k\ge2$ be an integer and $p$ an odd prime not dividing $k$. Then (\ref{rank1})
holds for the hypergeometric series
$${}_{k-1}F_{k-2}(\nicefrac1k,\nicefrac2k,\ldots,\nicefrac{(k-1)}{k};1,1,\ldots,1|t).$$
\end{corollary}

\begin{proof}
Consider 
$$g=k^{-k}\left(x_1+\cdots+x_{k-1}+\frac{1}{x_1\cdots x_{k-1}}\right).$$
A simple calcuation show that $f_r$ is zero if $k$ does not divide $r$ and equal to
$$k^{-kl}{kl\choose l}=\frac{(\nicefrac1k)_l}{l!}\frac{(\nicefrac2k)_l}{l!}
\cdots\frac{(\nicefrac{(k-1)}{k})_l}{l!}
$$
if $r=kl$. Hence
$$F(t)={}_{k-1}F_{k-2}(\nicefrac1k,\nicefrac2k,\ldots,\nicefrac{(k-1)}{k};1,1,\ldots,1|t^k).$$
Now apply Theorem \ref{BVthm} with $m=k$ and replace $t^k$ by $t$.
\end{proof}
Here is another variation which generalizes Dwork's example
\begin{corollary}
Let $k\ge2$ be an integer and $p$ an odd prime. Then (\ref{rank1}) holds
for the hypergeometric series
$${}_{k-1}F_{k-2}(\nicefrac12,\nicefrac12,\ldots,\nicefrac12,1,1,\ldots,1|t).$$
\end{corollary}

\begin{proof}
Consider
$$g=4^{-k}\left(x_1+\frac{1}{x_1}\right)\cdots\left(x_k+\frac{1}{x_k}\right).$$
A simple calculation shows that $f_r$ is zero if $r$ is odd and equal to
$$\left(\frac{(\nicefrac12)_l}{l!}\right)^k$$
if $r=2l$. Hence
$$F(t)={}_{k-1}F_{k-2}(\nicefrac12,\ldots,\nicefrac12;1,1,\ldots,1|t^2).$$
Now apply Theorem \ref{BVthm} with $m=2$ and replace $t^2$ by $t$.
\end{proof}

Having seen the above examples one can mix the ideas by taking combinations for
$g$, for example $g=(x+1/x)(y_1+y_2+1/y_1y_2)$ yielding 
$${}_3F_2(\nicefrac12,\nicefrac13,\nicefrac23,1,1|t).$$

\section{One variable polynomials}
Let again $R$ be a characteristic zero ring, $p$ an odd prime such that
$\cap_{s\ge1} p^sR=\{0\}$ and suppose $R$ is $p$-adically complete.
Let $\sigma:R\to R$ be a Frobenius lift.
It turns out that in the case of one variable polynomials $f$ the theory
sketched in Section \ref{theory} has a very
nice simplication that we like to present for general monic $f\in R[x]$
with $f(0)\ne0$. Let $d$ be
the degree of $f$. We suppose that $d\ge2$ and that the discriminant
of $f$ is invertible in $R$. 

The space $\Omega_f^\circ$ is given by $\scrO_f^\circ dx$ where
$\scrO_f^\circ$ is the $R$-module generated by the forms $l!\frac{x^k}{f^{l+1}}$
with $0\le k\le d(l+1)-2$. Similarly we define $\scrO_f$ in the same way but with the
inequalities $0\le k\le d(l+1)-1$. The exact forms in $\Omega_f^\circ$ are then
given by $d\scrO_f$. We call them {\it rational exact forms}.

We define $\scrO_{\rm formal}=\frac{1}{x}R[[1/x]]$ and 
$\Omega_{\rm formal}=\frac{1}{x}\scrO_{\rm formal}dx$.
We embed $\Omega_f^\circ$ in $\Omega_{\rm formal}$ by expansion in powers of $1/x$. 
The {\it formally exact forms} are defined by $d\scrO_{\rm formal}$.

The interior of the Newton polytope is $\Delta^\circ=(0,d)$ and the cardinality of
$\Delta^\circ\cap\Z$ is $d-1$. So, letting $p$ be an odd prime,
the Hasse-Witt matrix $\beta_p(t)$ is a $(d-1)\times (d-1)$-matrix. It turns out that
$\det(\beta_p)\is {\rm disc}(f)^{p-1}\mod{p}$, where
${\rm disc}(f)$ is the discriminant of $f$.
By $p$-adic completeness of $R$ and invertibility of ${\rm disc}(f)$ in $R$
we find that $\det(\beta_p)$ is invertible in $R$.
According to Theorem 9 in
Dwork crystals I, \cite{BV19I}, we know that $\hat\Omega_f^\circ/d\scrO_{\rm formal}$
is a free rank $d-1$ module over $R$ with basis $dx/f,xdx/f,\ldots,x^{d-2}dx/f$.
 
It turns out that in the case $n=1$ formally exact
forms coincide with rational exact forms. More precisely,

\begin{proposition}\label{formal2rational}
Let $f\in R[x]$ be a monic polynomial and suppose that its discriminant is invertible
in $R$. Then $\Omega_f^\circ\cap d\scrO_{\rm formal}= d\scrO_f$. 
\end{proposition}

\begin{proof}
Clearly $d\scrO_f\subset d\scrO_{\rm formal}$. We first show that every
$\omega\in\Omega_f^\circ$ is equivalent modulo $d\scrO_f$ to a form
$Q(x)dx/f$ with $Q(x)\in R[x]$ of degree
$\le d-2$. To that end we use the one variable version of Griffiths's reduction procedure.
Since $p$ does not divide ${\rm disc}(f)$, to every $Q(x)\in R[x]$ of degree $\le N$
there exist
polynomials $A,B\in R[x]$ of degrees $\le d-1$ and $\le\max(d-2,N-d)$ respectively,
such that $Q=Af'+Bf$. 

Let us start with a form $l!Q(x)dx/f^{l+1}$ with $\deg(Q)\le (l+1)d-2$ and $l>0$.
Write $Q=Af'+Bf$ with $\deg(A)\le d-1, \deg(B)\le ld-2$. Then we obtain
\begin{eqnarray*}
l!\frac{Q(x)}{f^{l+1}}dx&=&l!\frac{Af'}{f^{l+1}}dx+l!\frac{B}{f^l}dx\\
&=&d\left((l-1)!\frac{A}{f^l}\right)-(l-1)!\frac{A'}{f^l}dx+l!\frac{B}{f^l}dx\\
&\is&(l-1)!\frac{lB-A'}{f^l}dx.
\end{eqnarray*}
Note that $\deg(lB-A')\le ld-2$. By repeating this procedure we see that any
$\omega\in\Omega_f^\circ$
is equivalent modulo $d\scrO_f$ to a form $Qdx/f$ with $Q\in R[x]$ of degree
$\le d-2$.

The second part of our proof consists of showing that $Qdx/f\in d\scrO_{\rm formal}$ implies
that $Q=0$. Suppose that 
$$\frac{Qdx}{f}=d\left(\sum_{n\ge0}\frac{a_n}{x^n}\right)=
\sum_{n\ge1}-\frac{na_n}{x^{n+1}}dx.$$
From this we see that the coefficient of $dx/x^{mp^s+1}$ in the $1/x$-expansion of $Qdx/f$
is divisible by $p^s$ for any $m,s\ge0$. Let $K$ be the splitting field of $f$ over
$R$ and let $\alpha_1,\ldots,\alpha_d\in K$ be the zeros of $f$. Then there exist 
$A_1,\ldots,A_d$ in $R[\alpha_1,\ldots,\alpha_d]$ such that
$${\rm disc}(f)\frac{Qdx}{f}=\sum_{i=1}^d\frac{A_idx}{x-\alpha_i}=
\sum_{n\ge0}(A_1\alpha_1^n+\cdots+A_d\alpha_d^n)\frac{dx}{x^{n+1}}.$$
We now know that $A_1\alpha_1^{mp^s}+\cdots+A_d\alpha_d^{mp^s}$ is divisible by $p^s$
for all $m\ge0$. In particular for $m=0,1,\ldots,d-1$. Now note that
\begin{eqnarray*}
\det((\alpha_i^{mp^s})_{i=1,\ldots,d;m=0,\ldots,d-1})
&=&\prod_{i<j}(\alpha_i^{p^s}-\alpha_j^{p^s})\\
&\is& \prod_{i<j}(\alpha_i-\alpha_j)^{p^s}\is {\rm disc}(f)^{p^s}\mod{p},
\end{eqnarray*}
which is a unit in $R$. We conclude that $A_i\is0\mod{p^s}$ for all $i$ and $s$.
Hence $A_i=0$ for all $i$ and we conclude $Q(x)=0$, as asserted.
\end{proof}

An immediate corollary is its extension to $p$-adic completions. Denote
$\hat\Omega_f^\circ$ as before and similarly $\hat\scrO_f$. Then we find,

\begin{proposition}\label{formal2rationalhat}
Let $f\in R[x]$ be a monic polynomial and suppose that its discriminant is invertible
in $R$. Then $U_f^\circ=\hat\Omega_f^\circ\cap d\scrO_{\rm formal}=d\hat\scrO_f$. 
\end{proposition}

The operator $\cartier$ is essentially a
lift of a Cartier operator which is only well-defined in characteristic $p$.
In \cite{BV19I} and \cite{BV19II} it sufficed to use only the operator $\cartier$ defined
above. However, as a new ingredient, we need to consider other lifts.
Let $a\in\Z_p$. Define $\cartier^a$ as the operator 
with the property that $\cartier^a((x-a)^{k-1}dx)=(x-a)^{k/p-1}dx$ if $p$ 
divides $k$ and $0$ if not. In general it acts on rational differential forms as
$$\cartier^a\left(S(x)\frac{dx}{x}\right)=\sum_{y:(y-a)^p=x-a}S(y)\frac{dy}{y}.$$
So we sum over $y=a+\zeta(x-a)^{1/p}$ where $\zeta$ runs over the $p$-th roots of unity.
We can compare $\cartier$ and $\cartier^a$ by looking at their action on 
$\Omega_{\rm formal}$. 

\begin{proposition}
We have $\cartier^a(\Omega_f^\circ)\subset \hat\Omega_{f^\sigma}^\circ$ and 
\begin{equation}\label{lift-invariance}
\cartier(\omega)\is\cartier^a(\omega)\mod{pd\hat\scrO_{f^\sigma}}
\end{equation}
for all $\omega\in\Omega_f^\circ$.
\end{proposition}

\begin{proof}
The fact that the image of $\cartier^a$ lies in $\hat\Omega_{f^\sigma}^\circ$ follows
along the same lines as in the proof of \cite[Prop 5]{BV19I}. Clearly we have
$R[[1/x]]\cong R[[1/(x-a)]]$ through the expansion $\frac{1}{x-a}=
\sum_{n\ge0}\frac{a^n}{x^{n+1}}$. 
Let us prove our second assertion for $\omega_k=(x-a)^{-k-1}dx$ for $k\ge1$.
The full statement then follows by linearity.

Observe that
$$\omega_k=(x-a)^{-k-1}dx=-d\left(\frac1k(x-a)^{-k}\right).$$
If $k$ is not divisible by $p$ then clearly $\omega_k\in d\scrO_{\rm formal}$.
Since $\cartier(d\scrO_{\rm formal})\subset pd\scrO_{\rm formal}$ we get that 
$\cartier(\omega_k)\is 0\mod{pd\scrO_{\rm formal}}$. We have trivially
$\cartier^a(\omega_k)=0$. This proves our statement for $k$ not divisible by $p$.
Suppose now that $p$ divides $k$. Then 
$$\frac{1}{k}(x-a)^{-k}\is\frac{1}k(x^p-a)^{-k/p}\mod{\scrO_{\rm formal}}$$
hence, after taking differentials,
$$(x-a)^{-k-1}dx\is (x^p-a)^{-k/p-1}x^{p-1}dx\mod{d\scrO_{\rm formal}}.$$
Application of $\cartier$ gives 
$\cartier(\omega_k)\is\omega_{k/p}\mod{pd\scrO_{\rm formal}}$. Note that
$\omega_{k/p}=\cartier^a(\omega_k)$ when $p$ divides $k$. Thus we conclude that
$$\cartier(\omega_k)\is \cartier^a(\omega_k)\mod{pd\scrO_{\rm formal}}.$$
By linearity this congruence holds for all $\omega\in\Omega_f^\circ$.

It remains to see that we can replace $pd\scrO_{\rm formal}$ by $pd\hat\scrO_f.$
From Proposition 7 in \cite{BV19I} it follows that to any $\omega\in\hat\Omega_f^\circ$
there exists $\omega_1\in\hat\Omega_{f^\sigma}^\circ$ and a polynomial $A(a,\omega)$ such that
$\cartier^a(\omega)=\frac{A(a,\omega)}{f^\sigma}+p\omega_1$. Since 
$\cartier^a(\omega)-\cartier^0(\omega)\in pd\scrO_{\rm formal}$ it follows that
$A(a,\omega)-A(0,\omega)$ is divisible by $p$. Hence 
$$\frac{1}{p}(\cartier^a(\omega)-\cartier^0(\omega))\in \hat\Omega_{f^\sigma}^\circ
\cap d\scrO_{\rm formal}=d\hat\scrO_{f^\sigma}.$$
The latter equality follows from Proposition \ref{formal2rationalhat}.
\end{proof}

\section{A matrix example}\label{matrixexample}
The examples in the Section \ref{MUM} are all related to the case $h=1$, one
interior lattice point of the Newton polytope $\Delta$.
In this section we consider an example of rank $h=2$. 

\begin{theorem}\label{main}
Let
$$\scrY(t)=\begin{pmatrix}
F(\nicefrac13,\nicefrac23,\nicefrac12|t^2)&
-\frac1{3}tF(\nicefrac76,\nicefrac56,\nicefrac32|t^2)\\
-\frac{2}{3}tF(\nicefrac23,\nicefrac43,\nicefrac32|t^2)&
F(\nicefrac16,\nicefrac56,\nicefrac12|t^2)
\end{pmatrix}.$$
Denote by $\scrY_m(t)$ the $m$-th truncated version of $\scrY(t)$, i.e. we
drop all term starting with $t^m$. Then, for all primes $p>3$ and all $m,s\ge1$
we have
$$\scrY_{mp^s}(t)\begin{pmatrix}\epsilon_p & 0\\0&1\end{pmatrix}\scrY_{mp^{s-1}}(t^p)^{-1}\is
\scrY(t)\begin{pmatrix}\epsilon_p & 0\\0&1\end{pmatrix}\scrY(t^p)^{-1}\mod{p^s}.
$$
Here $\epsilon_p=1$ if $3$ is a square modulo $p$ and $-1$ if not. 
\end{theorem}

For the proof of this theorem, given at the end of this section,
we require the one variable polynomial $f=x^3-x-t\in R[x]$
with $R=\Z_p[[t]]$, where $p$ is a prime with $p>3$.
As Frobenius lift we take $g(t)^\sigma=g(t^p)$ for all $g(t)\in R$.
The discriminant of $f$ equals $4-27t^2$, and hence it is invertible in $R$.

We define the $2\times2$-matrix $\Lambda_p$ with entries in $R$ by
\begin{equation}\label{cartieraction}
\cartier\begin{pmatrix}dx/f \\ xdx/f\end{pmatrix}\is
\Lambda_p\begin{pmatrix}dx/f^\sigma \\ xdx/f^\sigma\end{pmatrix}\mod{d\hat\scrO_f}.
\end{equation}

The relation of $\Lambda_p$ with hypergeometric functions is obtained by period maps.
To that end we
consider
$$l!\frac{x^{k-1}dx}{f^{l+1}}=
l!\frac{x^{k-1}dx}{(x^3-x)^{l+1}}\sum_{r\ge0}{r+l\choose l}
\frac{t^r}{(x^3-x)^r},$$
and then take termwise the residue
at $x=0$. We could rephrase this procedure by saying that we expand $x^{k-1}dx/f^{l+1}$
as two-sided Laurent
series in $R[[x,t/x]]$ and then take the residue at $x=0$. Similarly we can
take residues at $x=\pm1$ (i.e. by expanding in Laurent series in $x\mp1$).
The result is again a power series in $t$. As long as $0<k<3(l+1)$, the sum of these three
series is $0$ because the termwise residues have sum $0$. We carry out the residue
computations for $l=0,k=1,2$. A straightforward calculation
shows that
$${\rm res}_{x=0}\ \frac{dx}{(x^3-x)^{r+1}}=\begin{cases}0 & \mbox{if $r$ is odd}\\
-{3n\choose n}& \mbox{if $r=2n$}\end{cases}$$

$${\rm res}_{x=0}\ \frac{xdx}{(x^3-x)^{r+1}}=\begin{cases}0 & \mbox{if $r$ is even}\\
{3n+1\choose n}& \mbox{if $r=2n+1$}\end{cases}.$$

Denote ${\rm res}_{\pm}\omega={\rm res}_{x=1}\omega-{\rm res}_{x=-1}\omega$.
Then we obtain

$${\rm res}_{\pm}\ \frac{dx}{(x^3-x)^{r+1}}=\begin{cases}0 & \mbox{if $r$ is even}\\
-\frac32\ \frac{(7/6)_n(5/6)_n}{(3/2)_n n!}
\left(\frac{27}4\right)^{n}& \mbox{if $r=2n+1$}\end{cases}$$

$${\rm res}_{\pm}\ \frac{xdx}{(x^3-x)^{r+1}}=\begin{cases}0 & \mbox{if $r$ is odd}\\
\frac{(1/6)_n(5/6)_n}{(1/2)_n n!}
\left(\frac{27}4\right)^{n}& \mbox{if $r=2n$}\end{cases}.$$

Let us denote the period map obtained by taking {\it minus} the residue at $0$ by $\rho_0$
and the one by taking residues at $\pm1$ by $\rho_\pm$. We summarize
$$\rho_0\left(dx/f\right)=F(\nicefrac13,\nicefrac23,\nicefrac12|
\nicefrac{27t^2}4).$$
$$\rho_0\left(xdx/f\right)=-t
F(\nicefrac23,\nicefrac43,\nicefrac32|\nicefrac{27t^2}4).$$
$$\rho_\pm\left(dx/f\right)=-\frac32 t
F(\nicefrac76,\nicefrac56,\nicefrac32|\nicefrac{27t^2}4).$$
$$\rho_\pm\left(xdx/f\right)=
F(\nicefrac16,\nicefrac56,\nicefrac12|\nicefrac{27t^2}4).$$

A crucial property of $\rho_0,\rho_{\pm}$ is that they vanish on exact forms,
i.e. $d\hat\scrO_f$. This is because residues of exact forms are zero, which 
is a special case of \cite[Prop 2]{BV19II}.

\begin{proposition}
For every $\omega\in\hat\Omega_f^\circ$ we have $\rho_0(\cartier(\omega))=
\rho_0(\omega)$ and $\rho_\pm(\cartier(\omega))=\rho_\pm(\omega)$.
\end{proposition}

\begin{proof}

Let $\omega\in\hat\Omega_f^\circ$. Expand it in $R[[x,t/x]]dx$. The value of $\rho_0$ 
is minus the coefficient of $dx/x$. By definition of $\cartier$ this value is
the same for $\cartier(\omega)$, hence our first assertion follows.
Similarly we can see that $\rho_1$, the residue at $1$, has the property
$\rho_1(\cartier^1(\omega))=\rho_1(\omega)$. It follows from Proposition
\ref{lift-invariance} that $\cartier^1(\omega)\is\cartier(\omega)\mod{d\hat\scrO_f}$.
Hence $\rho_1(\cartier(\omega))=\rho_1(\omega)$. The same result holds of course
for $\rho_\pm=\rho_1-\rho_{-1}$. 
\end{proof}

\begin{corollary}\label{dworklimit}
Let 
$$Y(t)=\begin{pmatrix}
F(\nicefrac13,\nicefrac23,\nicefrac12|\nicefrac{27t^2}4)&
-\frac32 tF(\nicefrac76,\nicefrac56,\nicefrac32|\nicefrac{27t^2}4)\\
-tF(\nicefrac23,\nicefrac43,\nicefrac32|\nicefrac{27t^2}4)&
F(\nicefrac16,\nicefrac56,\nicefrac12|\nicefrac{27t^2}4)
\end{pmatrix}.$$
Let $\Lambda_p$ be the $2\times2$ cartier-matrix in (\ref{cartieraction}). Then
$$\Lambda_p=Y(t)Y(t^p)^{-1}.$$ 
\end{corollary}

\begin{proof}
We start with the equality (\ref{cartieraction}), apply $\rho_0$ and use
$\rho_0\circ\cartier=\rho_0$ to obtain
$$
\begin{pmatrix}\rho_0(dx/f) \\ \rho_0(xdx/f)\end{pmatrix}=
\Lambda_p\begin{pmatrix}\rho_0(dx/f^\sigma) \\ \rho_0(xdx/f^\sigma)\end{pmatrix}.
$$
Similarly we obtain
$$
\begin{pmatrix}\rho_\pm(dx/f) \\ \rho_\pm(xdx/f)\end{pmatrix}=
\Lambda_p\begin{pmatrix}\rho_\pm(dx/f^\sigma) \\ \rho_\pm(xdx/f^\sigma)\end{pmatrix}.
$$
Our corollary follows from the above evaluations of the periods.
\end{proof}

In order to get Dwork type congruences we also need to introduce a suitable
'period map mod $m$'. By that we mean an $R$-linear map $\rho:\hat\Omega_f\to R$ such
that $\rho(\hat\Omega_f\cap d\scrO_{\rm formal})\subset mR$ and $\delta\circ\rho
\is \rho\circ\delta\mod{mR}$ for any derivation $\delta$ on $R$.

For our purposes we use a slight generalization of the period maps we considered in
\cite[Section 4]{BV19II}. We define $\rho_{0,m}$ by
$$\rho_{0,m}\omega = \rho_0\left(1-\frac{t^m}{(x^3-x)^m}\right)\omega.$$
Similarly we define $\rho_{\pm,m}$. 
As an illustration we elaborate $\rho_{0,m}(dx/f)$. We get
\begin{eqnarray*}
\rho_{0,m}(dx/f)&=&-{\rm res}_{x=0}\left(1-\frac{t^m}{(x^3-x)^m}\right)
\frac{dx}{x^3-x-t}\\
&=&-{\rm res}_{x=0}\frac{1}{(x^3-x)^{m}}\sum_{r=0}^{m-1}(x^3-x)^{m-1-r} t^r dx\\
&=&-{\rm res}_{x=0}\sum_{r=0}^{m-1}\frac{t^rdx}{(x^3-x)^{r+1}}\\
&=&\sum_{2n<m}{3n\choose n}t^{2n}.
\end{eqnarray*}
The latter polynomial is the truncation of $F(\nicefrac13,\nicefrac23,\nicefrac12|
\nicefrac{27t^2}{4})$ truncated at the degree $m$ term. Denote the truncation at
degree $m$ of a power series $g(t)$ by $g(t)_m$. Then we obtain
$$\rho_{0,m}\left(dx/f\right)=F(\nicefrac13,\nicefrac23,\nicefrac12|
\nicefrac{27t^2}4)_m.$$
$$\rho_{0,m}\left(xdx/f\right)=-(t
F(\nicefrac23,\nicefrac43,\nicefrac32|\nicefrac{27t^2}4))_m.$$
$$\rho_{\pm,m}\left(dx/f\right)=-\frac32 (t
F(\nicefrac76,\nicefrac56,\nicefrac32|\nicefrac{27t^2}4))_m.$$
$$\rho_{\pm,m}\left(xdx/f\right)=
F(\nicefrac16,\nicefrac56,\nicefrac12|\nicefrac{27t^2}4)_m.$$

\begin{lemma}
We have $\rho_{0,m}(d\hat\scrO_f)\is0\mod{m}$ and $\rho_{1,m}(d\hat\scrO_f)\is0\mod{m}$.

Secondly, for any $m\ge1$ divisible by $p$ we have $\rho_{0,m}\is\rho^\sigma_{0,m/p}\circ
\cartier\mod{p^{\ord_p(m)}}$ and $\rho_{\pm,m}\is\rho^\sigma_{\pm,m/p}\circ
\cartier\mod{p^{\ord_p(m)}}$.
\end{lemma}

\begin{proof}
For any $G\in\hat\scrO_f$ we have 
\begin{eqnarray*}
\rho_{0,m}dG&=&-\mbox{coefficient of $\frac{dx}{x}$ in }
\left(1-\left(\frac{t}{x^3-x}\right)^m\right)dG\\
&\is&-\mbox{coefficient of $\frac{dx}{x}$ in }
d\left(1-\left(\frac{t}{x^3-x}\right)^m\right)G\is0\mod{m}.
\end{eqnarray*}
The applicability of $\rho_0$ requires that we consider expansions as doubly
infinite Laurent series in $R[[x,t/x]]$. 
For $\rho_{1,m}$ the proof runs similarly.

For the proof of the second part let $\omega\in\hat\Omega_f$. Then we have
\begin{eqnarray*}
\rho_{0,m}(\omega)&=&
-\mbox{coefficient of $\frac{dx}{x}$ in }
\left(1-\left(\frac{t}{x^3-x}\right)^m\right)\omega\\
&\is&-\mbox{coefficient of $\frac{dx}{x}$ in }
\left(1-\left(\frac{t^p}{x^{3p}-x^p}\right)^{m/p}\right)\omega
\ \mod{p^{\ord_p(m)}}\\
&\is&-\mbox{coefficient of $\frac{dx}{x}$ in }
\cartier\left(1-\left(\frac{t^p}{x^{3p}-x^p}\right)^{m/p}
\right)\omega\ \mod{p^{\ord_p(m)}}\\
&\is&-\mbox{coefficient of $\frac{dx}{x}$ in }
 \left(1-\left(\frac{t^p}{x^3-x}\right)^{m/p}\right)
\cartier(\omega)\ \mod{p^{\ord_p(m)}}\\
&\is& \rho^\sigma_{0,m/p}\cartier(\omega)\ \mod{p^{\ord_p(m)}}.
\end{eqnarray*}
The second step uses the obvious fact that the Cartier transform does not change
the coefficient of $\frac{dx}{x}$.

In a similar manner one can show that 
$$\rho_{1,m}(\omega)=\rho^\sigma_{1,m/p}\cartier^1(\omega)\mod{p^{\ord_p(m)}}.$$
Proposition \ref{lift-invariance} tells us that $\cartier^1(\omega)\is\cartier(\omega)
\mod{pd\hat\scrO_{f^\sigma}}$. Together with the first part of our proposition, which implies
that $\rho_{1,m/p}^\sigma(pd\hat\scrO_{f^\sigma})\is0\mod{p^{\ord_p(m)}}$, we get
$$\rho_{1,m}(\omega)\is\rho^\sigma_{1,m/p}\cartier(\omega)\mod{p^{\ord_p(m)}}.$$
In a similar way the statement for $\rho_{\pm,m}$ follows.
\end{proof}

\begin{corollary}\label{dworkmodm}
Let notations be as in Corollary \ref{dworklimit}
Let $Y(t)_m$ be the matrix $Y(t)$, where the entries have been truncated at $t^m$.
Then, for any $m,s\ge1$,
$$Y(t)_{mp^s}\is(Y(t)Y(t^p)^{-1})Y(t^p)_{mp^{s-1}}\mod{p^s}.$$ 
\end{corollary}

\begin{proof}
We start with the equality (\ref{cartieraction}), apply $\rho_{0,mp^{s-1}}$ and use
$\rho_{0,mp^s}\is\rho_{0,mp^{s-1}}^\sigma\circ\cartier\mod{p^s}$ to obtain
$$
\begin{pmatrix}\rho_{0,mp^s}(dx/f) \\ \rho_{0,mp^s}(xdx/f)\end{pmatrix}\is
\Lambda_p\begin{pmatrix}\rho_{0,mp^{s-1}}^\sigma(dx/f^\sigma) \\ 
\rho_{0,mp^{s-1}}^\sigma(xdx/f^\sigma)\end{pmatrix}\mod{p^s}.
$$
Similarly we obtain
$$
\begin{pmatrix}\rho_{\pm,mp^s}(dx/f) \\ \rho_{\pm,mp^s}(xdx/f)\end{pmatrix}\is
\Lambda_p\begin{pmatrix}\rho_{\pm,mp^{s-1}}^\sigma(dx/f^\sigma) \\ 
\rho_{\pm,mp^{s-1}}^\sigma(xdx/f^\sigma)\end{pmatrix}\mod{p^s}.
$$
Our corollary follows from the above evaluations of the mod $m$ periods
and $\Lambda_p=Y(t)Y(t^p)^{-1}$.
\end{proof}

We end with the proof of our main theorem.

\begin{proof}[Proof of Theorem \ref{main}]
The proof follows the same steps as Corollary \ref{dworklimit}, but with the
polynomial $f=x^3-x-2t/3\sqrt3$. This polynomial is defined over $\Z_p[\sqrt3][[t]]$
with Frobenius lift $\sigma$ such that $\sigma(t)=t^p$ and $\sigma(\sqrt3)=\epsilon_p
\sqrt3$. Hence $f^\sigma=x^3-x-2\epsilon_pt^p/\sqrt3$. We also use the new basis
$dx/f,\sqrt3 xdx/f$ and replace $\rho_\pm$ by $\frac{1}{\sqrt3}\rho_\pm$. 
The adapted version of Corollary \ref{dworklimit} would then become
$$\Lambda_p=\scrY(t)\begin{pmatrix}\epsilon_p & 0\\0&1\end{pmatrix}\scrY(t^p)^{-1}.$$
The remainder of the proof follows the same lines as above.
\end{proof}

We finally give, without proof, the system of differential equations for $\scrY(t)$ and its
congruence version. Again the proof follows the same lines as in \cite{BV19II}.

\begin{theorem}
We have
$$\frac{d}{dt}\scrY(t)=\frac{1}{3(1-t^2)}\begin{pmatrix}2t & -1\\-2 & t\end{pmatrix}
\scrY(t)$$
and
$$\frac{d}{dt}\scrY_{mp^s}(t)\is\frac{1}{3(1-t^2)}\begin{pmatrix}2t & -1\\-2 & t\end{pmatrix}
\scrY_{mp^s}(t)\mod{p^s}$$
For all $m,s\ge1$.
\end{theorem}

\end{document}